\documentclass[a4paper,12pt]{amsart}
\usepackage{amsmath,amsthm,amsfonts,amssymb}
\usepackage{indentfirst}
\usepackage[numbers]{natbib}
\usepackage{a4wide}
\usepackage{graphicx}
\usepackage{hyperref}

\theoremstyle{plain}
\newtheorem{teo}{Theorem}[section]

\newtheorem{lem}[teo]{Lemma}
\newtheorem{defn}[teo]{Definition}

\theoremstyle{remark}
\newtheorem{obs}[teo]{Remark}
\newtheorem{exa}[teo]{Example}

\numberwithin{equation}{section}

\newcommand{\bbR}{\mathbb{R}}

\newcommand{\eps}{\varepsilon}

\newcommand{\Cov}{\textup{Cov}}

\newcommand{\al}{\alpha}
\newcommand{\la}{\lambda}
\newcommand{\de}{\delta}
\newcommand{\te}{\theta}
\newcommand{\tu}{\theta_1}
\newcommand{\td}{\theta_2}
\newcommand{\ga}{\gamma}
\newcommand{\ka}{\kappa}
\newcommand{\xf}{x_{\infty}}
\newcommand{\uf}{u_{\infty}}
\newcommand{\tf}{\tau_{\infty}}
\newcommand{\Si}{\Sigma}
\newcommand{\vn}{\sigma^2}
\newcommand{\La}{\Lambda}
\newcommand{\apq}{(\al, p, q)}
\newcommand{\GV}{\mathcal{V}}

\newcommand{\LW}[1]{\ensuremath{W_0(#1)}}
\newcommand{\LWR}[1]{\ensuremath{W_{-1}(#1)}}

\begin{document}

\title[Stochastic rumour model]{Limit theorems for a general \\
stochastic rumour model}

\author[Elcio Lebensztayn et al.]{Elcio Lebensztayn}
\address[E. Lebensztayn and F. P. Machado]{Statistics Department, Institute of Mathematics and Statistics, University of S\~ao Paulo, CEP 05508-090, S\~ao Paulo, SP, Brazil.}
\email{elcio@ime.usp.br}
\thanks{Elcio Lebensztayn was supported by CNPq (311909/2009-4), F\'abio P. Machado by CNPq (306927/2007-1) and FAPESP (2009/18253-7), and Pablo M. Rodr\'iguez by FAPESP (2010/06967-2).}
\author[]{F\'abio P. Machado}
\email{fmachado@ime.usp.br}
\author[]{Pablo M. Rodr\'iguez}
\address[P. M. Rodr\'{i}guez]{Statistics Department, Institute of Mathematics, Statistics and Scientific Computation, University of Campinas, CEP 13083-859, Campinas, SP, Brazil.}
\email{pablor@ime.unicamp.br}

\date{\today}

\begin{abstract}
We study a general stochastic rumour model in which an ignorant individual has a certain probability of becoming a stifler immediately upon hearing the rumour. 
We refer to this special kind of stifler as an uninterested individual. 
Our model also includes distinct rates for meetings between two spreaders in which both become stiflers or only one does, so that particular cases are the classical Daley--Kendall and Maki--Thompson models. 
We prove a Law of Large Numbers and a Central Limit Theorem for the proportions of those who ultimately remain ignorant and those who have heard the rumour but become uninterested in it.
\end{abstract}

\maketitle 

\section{Introduction}
\label{S: Introduction}

Various mathematical models for the propagation of a rumour within a population have been developed in the past decades. Two classical models were introduced by Daley and Kendall~\cite{DK} and Maki and Thompson~\cite{MT}. 
In the first model (which we shorten to [DK] model), a closed homogeneously mixing population of $N + 1$ individuals is considered.
People are subdivided into three classes: \textit{ignorants} (those not aware of the rumour), \textit{spreaders} (who are spreading it), and \textit{stiflers} (who know the rumour but have ceased communicating it after meeting somebody who has already heard it). 
We adhere to the usual notation, denoting the number of ignorants, spreaders and stiflers at time $t$ by $X(t)$, $Y(t)$ and $Z(t)$, respectively.
Initially, $X(0) = N$, $Y(0) = 1$ and $Z(0) = 0$, and $X(t) + Y(t) + Z(t) = N + 1$ for all~$t$.
The process $\{(X(t), Y(t))\}_{t \geq 0}$ is a continuous-time Markov chain with transitions and corresponding rates given by
\begin{equation*}
{\allowdisplaybreaks
\begin{array}{cc}
\text{transition} \quad &\text{rate} \\[0.1cm]
(-1, 1) \quad &X Y, \\[0.1cm]
(0, -2) \quad &\displaystyle\binom{Y}{2}, \\[0.2cm]
(0, -1) \quad &Y (N + 1 - X - Y).
\end{array}}%
\end{equation*}
People interact by pairwise contacts, and the three possible transitions correspond to spreader-ignorant, spreader-spreader and spreader-stifler interactions.
In the first case, the spreader tells the rumour to an ignorant, who becomes a spreader.
The two other transitions represent the transformation of the spreader(s) involved in the meeting into stifler(s), that is, the loss of interest in propagating the rumour derived from learning that it is already known by the other individual in the meeting.

In the model formulated by Maki and Thompson~\cite{MT} (which we shorten to [MT] model), differently from the [DK] model, the rumour is spread by \textit{directed contact} of the spreaders with other individuals, i.e., the initiator is distinguished from the recipient.
In addition, when a spreader contacts another spreader, only the initiating one becomes a stifler.
Thus, the continuous-time Markov chain $\{(X(t), Y(t))\}_{t \geq 0}$ evolves according to the following table
\begin{equation*}
\begin{array}{cc}
\text{transition} \quad &\text{rate} \\[0.1cm]
(-1, 1) \quad &X Y, \\[0.1cm]
(0, -1) \quad &Y (N - X).
\end{array}
\end{equation*}
We refer to Daley and Gani~\cite[Chapter~5]{DG} for an excellent account on the subject of rumour models.

We introduce a general stochastic rumour model with the following characteristics:
\begin{enumerate}
\item[(i)] Distinct rates for meetings between two spreaders in which both become stiflers or only one does.
\item[(ii)] A new class of individuals, which we call \textit{uninterested}.
In traditional models, once an ignorant is told the rumour, only the transformation into a spreader is allowed.
In our model, this individual has the choice (with a certain probability) of becoming uninterested, that is, of stifling right after hearing the rumour.
\end{enumerate}

We underline that a stifler is an ex-spreader who has lost interest in propagating the rumour, whereas an uninterested is an ex-ignorant whose lack of interest arose immediately upon hearing the rumour from a spreader.

To define the model, we consider the notation introduced so far and denote by $U(t)$ the number of uninterested individuals at time $t$.
Initially, $X(0) = N$, $U(0) = 0$, $Y(0) = 1$ and $Z(0) = 0$, and $X(t) + U(t) + Y(t) + Z(t) = N + 1$ for all~$t$.
Let $V(t) = (X(t), U(t), Y(t))$.
We suppose that $\{V(t)\}_{t \geq 0}$ is a continuous-time Markov chain with initial state $(N, 0, 1)$ and
\begin{equation*}
\begin{array}{cc}
\text{transition} \quad &\text{rate} \\[0.1cm]
(-1, 0, 1) \quad &\la \, \de \, X Y, \\[0.1cm]
(-1, 1, 0) \quad &\la (1 - \de) \, X Y, \\[0.1cm]
(0, 0, -2) \quad &\la \, \tu \displaystyle\binom{Y}{2}, \\[0.2cm]
(0, 0, -1) \quad &\la \, \td \, Y (Y - 1) + \la \, \ga \, Y (N + 1 - X - Y).
\end{array}
\end{equation*}
The first two cases correspond to a spreader telling the rumour to an ignorant, who decides to become a spreader or an uninterested (that is, an immediate stifler) with respective probabilities~$\de$ and $1 - \de$.
The interaction between two spreaders splits into two cases, in which both become stiflers or only one does.
These cases are represented by the third transition and by the first part of the rate of the fourth transition.
Finally, the second part of this rate corresponds to a meeting between a spreader and a stifler or an uninterested individual.

We define $\te = \tu + \td - \ga$ and assume throughout the paper that
\begin{equation}
\label{F: Hyp}
\la > 0, \, \ga > 0, \, \tu \geq 0, \, \td \geq 0, \, 0 < \de \leq 1 \, \text{ and } \, 0 \leq \te \leq 1.
\end{equation}
Notice that both [DK] and [MT] models can be obtained by suitably choosing these constants.
The family of Markov chains just defined includes the classical and other rumour models, some of which are presented in \S\ref{S: Main results}, showing the convenience of the adopted parametrization.

To state our results, it is convenient to write down explicitly the dependence of~$V$ on~$N$, so we use the notation $V^{(N)} = (X^{(N)}, U^{(N)}, Y^{(N)})$.
Observing that the process eventually terminates (when there are no spreaders in the population), we define
\[ \tau^{(N)} = \inf \{t: Y^{(N)}(t) = 0 \}. \]
In \S\ref{S: Main results}, we present a Weak Law of Large Numbers and a Central Limit Theorem for $ N^{-1} \, (X^{(N)}(\tau^{(N)}), U^{(N)}(\tau^{(N)})) $, that is, the ultimate fractions of the originally ignorant individuals who remained ignorant and who became uninterested. 
Proofs are presented in \S\ref{S: Proofs}.

As far as we know, limit theorems were rigorously proved only for the basic [DK] and [MT] models 
(see Pittel~\cite{Pittel}, Sudbury~\cite{Sudbury} and Watson~\cite{Watson}).
In addition, the main tools applied for stochastic rumour processes have been the analysis of the embedded Markov chain, martingale arguments, diffusion approximations, generating functions and the study of analogue deterministic versions.
A good survey of these methods can be found in Daley and Gani~\cite{DG}.
Briefly, the technique we use consists in defining a coupled process with the same transitions as the original process until absorption.
This construction is suitably defined, in such a way that the new process is a density dependent stochastic model for which the theory presented in Ethier and Kurtz~\cite{MPCC} can be applied. 
This approach allows us to establish limit theorems for more general models. 
To the best of our knowledge, this idea is used for the first time in the context of stochastic rumour models here and in Lebensztayn et al.~\cite{RPRS}, in which we study a generalization of the Maki--Thompson model with random stifling and general initial configuration.

\section{Main results}
\label{S: Main results}

We start off with a few definitions.

\begin{defn}
\label{D: xf}
For $0 < \te < 1$, consider the function $f: [0, 1] \rightarrow \bbR $ given by
\begin{equation*}
f(x) = \frac{(\ga + \de \te) x^{\te} - (\ga + \de) \te x - \ga (1 - \te)}{\te (1 - \te)}.
\end{equation*}
We define $\xf = \xf(\de, \ga, \te)$ as the unique root of $f(x) = 0$ in the interval $(0, 1)$.
\end{defn}

To justify the existence and uniqueness of the root, notice that $f(0) < 0$, $f(1) = 0$ and that $f$ is unimodal, with a global maximum at the point
\[ \left(\frac{\ga + \de \te}{\ga + \de}\right)^{\frac{1}{1 - \te}} \in (0, 1). \]
For $\te = 1/2$, an explicit formula for $\xf$ can be obtained, namely,
\[ \xf(\de, \ga, 1/2) = \left(\frac{\ga}{\ga + \de}\right)^2. \]

In the cases $\te = 0$ and $\te = 1$, $\xf$ is defined similarly.

\begin{defn}
\label{D: xf01}
Consider the functions $f_0$ and $f_1$ defined on $(0, 1]$ by
\begin{align*}
f_0(x) &= \lim_{\te \to 0^{+}} f(x) = (\ga + \de) (1 - x) + \ga \log x, \\
f_1(x) &= \lim_{\te \to 1^{-}} f(x) = - \ga (1 - x) - (\ga + \de) \, x \log x.
\end{align*}
For each $\te \in \{ 0, 1 \}$, we denote by $\xf(\de, \ga, \te)$ the unique root of $f_{\te}(x) = 0$ in $(0, 1)$.
\end{defn}

\begin{obs}
In the last case, $\xf$ can be written in terms of the Lambert $W$ function, which is the multivalued inverse of the function $ x \mapsto x \, e^x $.
Let $W_0$ and $W_{-1}$ be the principal and the lower real branches of the Lambert $W$ function, respectively; see Corless et al.~\cite{LW} for more details.
Defining $h = 1 + \de / \ga$, we have that
\begin{align*}
\xf(\de, \ga, 0) &= - {h}^{-1} \, \LW{- h \, e^{-h}} \quad \text{and} \\
\xf(\de, \ga, 1) &= - \left[{h} \, \LWR{- e^{-1 / h} / h}\right]^{-1}.
\end{align*}
\end{obs}

We prove that, whatever the value of $\te$, the following inequality holds:
\begin{equation}
\label{F: Ineq xf}
\xf(\de, \ga, \te) < \frac{\ga}{\ga + \de}.
\end{equation}
For this, it is enough to show that $f$ and $f_{\te}$ evaluated at the point $\ga / (\ga + \de)$ are strictly greater than zero.
For $\te \in \{ 0, 1 \}$, formula~\eqref{F: Ineq xf} is therefore an immediate consequence of a standard logarithm inequality:
\begin{equation}
\label{F: Log}
\frac{u - 1}{u} < \log u < u - 1 \, \text{ for all } 0 < u < 1.
\end{equation}
For $\te \in (0, 1)$, we use~\eqref{F: Log} to prove that the function $ \te \mapsto (\ga / (\ga + \de \te))^{1 / \te} $ is increasing, whence~\eqref{F: Ineq xf} follows.

We are now ready to state our main results.

\begin{teo}
\label{T: WLLN}
Assume~\eqref{F: Hyp} and let $\xf$ be given by Definition~\ref{D: xf} or~\ref{D: xf01} according as $\te \in (0, 1)$ or not.
Define $\uf = (1 - \de)(1 - \xf)$.
Then,
\begin{equation*}
\lim_{N \to \infty} \, \frac{X^{(N)}(\tau^{(N)})}{N} = \xf \, \text{ and } 
\lim_{N \to \infty} \, \frac{U^{(N)}(\tau^{(N)})}{N} = \uf \quad \text{in probability}. 
\end{equation*}
\end{teo}

\begin{teo}
\label{T: CLT}
Assume~\eqref{F: Hyp} and define
{\allowdisplaybreaks
\begin{align*}
\ka &= 3 \tu + 2 \td - 4 \ga, \quad
A = \frac{\xf}{\ga - (\ga + \de) \xf}, \quad 
B = \frac{\ga \de \uf}{\ga + \de \te}, \\[0.1cm]
C &= (\ga + \de)^2 \left(4 \de \te^2 - \ka (\ga + 2 \de \te)\right) \xf + \ka \ga (\ga + \de) (\ga + \de (2 \te - 1)) - 4 \de \ga^2 (1 - \te)^2, \\[0.1cm]
D &=
\left\{
\begin{array}{cl}
\dfrac{C (1 - \xf)}{2 (2 \te - 1) (\ga + \de \te)^2}	&\text{ if } \te \neq \dfrac{1}{2}, \\[0.2cm]
\dfrac{2 \ga \left[\ka \de (2 \ga + \de) - 2 \ga (\de - \ka (\ga + \de)) \log \left(\frac{\ga}{\ga + \de}\right)\right]}{(\ga + \de)^2}	&\text{ if } \te = \dfrac{1}{2}.
\end{array}	\right.
\end{align*}}%

Then,
\begin{equation}
\label{F: Biv CLT}
\sqrt{N} \left(\frac{X^{(N)}(\tau^{(N)})}{N} - \xf, \frac{U^{(N)}(\tau^{(N)})}{N} - \uf\right) 
\stackrel{\mathcal{D}}{\rightarrow} N_2(0, \Si) \, \text{ as } \, N \to \infty, 
\end{equation}
where $ \stackrel{\mathcal{D}}{\rightarrow} $ denotes convergence in distribution, and $ N_2(0, \Si) $ is the bivariate normal distribution with mean zero and covariance matrix
\begin{equation*}
\Si =
\begin{pmatrix}
\Si_{11} & \Si_{12} \\
\Si_{21} & \Si_{22}
\end{pmatrix}
\end{equation*}
whose elements are given by
\begin{equation}
\label{F: Cov Matrix}
\begin{aligned}
\Si_{11} &= \xf (1 - \xf) + A^2 D, \\
\Si_{12} &= \Si_{21} = - (1 - \de) \Si_{11} + A B, \\
\Si_{22} &= (1 - \de)^2 \Si_{11} + (1 - \de) (\de (1 - \xf) - 2 A B).
\end{aligned}
\end{equation}
\end{teo}

Of course, when $\de = 1$, we can omit the second component of the vector in the left-hand side of~\eqref{F: Biv CLT}, and we denote by $\vn$ ($= \Si_{11}$) the variance of the asymptotic mean zero Gaussian distribution.

\begin{exa}
Let $\rho \in [0, 1]$ and consider our model with the choice $\la = \de = \ga = 1$, $\tu = \rho$ and $\td = 1 - \rho$, so $\te = 0$.
Thus, the limiting proportion of ignorants and the variance of the asymptotic normal distribution in the CLT are given respectively by
\begin{gather*}
\xf = \xf(1, 1, 0) = - \dfrac{\LW{- 2 \, e^{-2}}}{2} \approx 0.203188, \quad \text{and} \\[0.1cm]
\vn = \frac{\xf (1 - \xf) \left(1 - 2 \, \xf + 2 \, \rho \, \xf^2\right)}{(1 - 2 \, \xf)^2} \approx 
0.272736 + 0.0379364 \, \rho.
\end{gather*}
We obtain the [MT] or [DK] model according to whether $\rho$ equals $0$ or $1$, showing that our theorems generalize classical results presented by Sudbury~\cite{Sudbury} and Watson~\cite{Watson}.
\end{exa}

\begin{exa}
\label{E: Hayes}
In an interesting confessional essay, Hayes~\cite{Hayes} describes a mistake committed when he attempted to simulate the basic [DK] model.
The author actually simulated the [MT] model, with the difference that, when two spreaders meet, both become stiflers.
This model is obtained by choosing $\la = \de = \ga = 1$, $\tu = 2$ and $\td = 0$, in which case $\te = 1$, 
\begin{gather*}
\xf = \xf(1, 1, 1) = - \dfrac{1}{2 \, \LWR{- e^{-1 / 2} / 2}} \approx 0.284668, \quad \text{and} \\[0.1cm]
\vn = \frac{\xf (1 - \xf) \left(1 - 3 \, \xf + 3 \, \xf^2\right)}{(1 - 2 \, \xf)^2} \approx 0.427204.
\end{gather*}
This clarifies the numerical value of the proportion of ignorants remaining in the population that Hayes obtained in his simulations.
\end{exa}

\begin{exa}
\label{E: apq DK}
Let $\al, p, q \in (0, 1]$.
We describe a new variant of the [DK] model, which we call the $\apq$-[DK] model.
Suppose that, independently for each pairwise meeting and each individual,
\begin{enumerate}
\item[(a)] A spreader involved in a meeting decides to tell the rumour with probability~$p$.
\item[(b)] Once such a decision is made, any spreader in a meeting with somebody informed of the rumour has probability~$\al$ of becoming a stifler.
\item[(c)] Upon hearing the rumour, an ignorant becomes a spreader or an uninterested individual with respective probabilities~$q$ and $1 - q$.
\end{enumerate}

This model with $q = 1$ (no uninterested individuals) was introduced by Daley and Kendall~\cite{DK}, who studied its deterministic analogue.
This analysis is also presented in Daley and Gani~\cite[Section~5.2]{DG}.
The basic [DK] model corresponds to $\al = p = q = 1$.

Observe that, for the $\apq$-[DK] model, the continuous-time Markov chain $\{V(t)\}_{t \geq 0}$ evolves according to
\begin{equation*}
{\allowdisplaybreaks
\begin{array}{cc}
\text{transition} \quad &\text{rate} \\[0.1cm]
(-1, 0, 1) \quad &p q \, X Y, \\[0.1cm]
(-1, 1, 0) \quad &p (1 - q) \, X Y, \\[0.1cm]
(0, 0, -2) \quad &\al^2 p (2 - p) \displaystyle\binom{Y}{2}, \\[0.2cm]
(0, 0, -1) \quad &\al (1 - \al) p (2 - p) \, Y (Y - 1) + \al p \, Y (N + 1 - X - Y).
\end{array}}%
\end{equation*}
Therefore, Theorems~\ref{T: WLLN} and~\ref{T: CLT} yield novel bivariate limit theorems for this model, which are obtained by making the following substitutions:
\begin{equation*}
\la = p, \, \de = q, \, \tu = \al^2 (2 - p), \, \td = \al (1 - \al) (2 - p), \, \ga = \al \, \text{ and } \, \te = \al (1 - p).
\end{equation*}
Here are some important cases:

\smallskip
\textbf{(1)} For the $\apq$-[DK] model with $p < 1$, the asymptotic proportion of ignorants~$\xf$ is the unique root of
\[ f^{\ast}(x) = \frac{(1 + q (1 - p)) x^{\al (1 - p)} - (\al + q) (1 - p) x - 1 + \al (1 - p)}{(1 - p) (1 - \al (1 - p))} = 0 \]
in the interval $(0, 1)$.
The limiting fraction of uninterested individuals is $\uf = (1 - q)(1 - \xf)$.
When $q = 1$, the value $\xf$ coincides with that obtained for the deterministic analogue of the model in Daley and Gani~\cite[Equation~(5.2.8)]{DG}.

\smallskip
\textbf{(2)} The $(1, 1, q)$-[DK] model is the basic [DK] model with the additional rule that an ignorant is allowed not to have interest in spreading the rumour.
For this model, the limiting proportions of ignorant and uninterested individuals are expressed respectively as
\begin{equation}
\label{F: xu 11q DK}
\begin{aligned}
\xf &= \xf(q, 1, 0) = - \dfrac{\LW{- h \, e^{-h}}}{h} \text{ with } h = 1 + q \quad \text{and} \\[0.1cm] 
\uf &= (1 - q)(1 - \xf).
\end{aligned}
\end{equation}
The entries of the covariance matrix~$\Si$ given in~\eqref{F: Cov Matrix} simplify to
{\allowdisplaybreaks
\begin{align*}
\Si_{11} &= \frac{\xf (1 - \xf) \left( 2 - (3 + q^2) \xf + (1 + q)^2 \xf^2 \right)}{2 (1 - (1 + q) \xf)^2}, \\[0.1cm]
\Si_{12} &= \frac{\xf \, \uf \left( -2 (1 - q) + (1 - q) (3 + q) \xf - (1 + q)^2 \xf^2 \right)}{2 (1 - (1 + q) \xf)^2}, \\[0.1cm]
\Si_{22} &= \frac{\uf \left( 2 q + 2 (1 - 5 q) \xf + (-3 + 9 q + 3 q^2 - q^3) \xf^2 + (1 - q) (1 + q)^2 \xf^3 \right)}{2 (1 - (1 + q) \xf)^2}.
\end{align*}}%
When $q = 1$, these formulae reduce to the well-known results for the basic [DK] model.

\smallskip
\textbf{(3)} For the $(\al, 1, 1)$-[DK] model, we have that
\begin{gather*}
\xf = \xf(1, \al, 0) = - \dfrac{\LW{- h \, e^{-h}}}{h} \text{ with } h = 1 + \frac{1}{\al}, \quad \text{and} \\[0.1cm]
\vn = \frac{\xf (1 - \xf) \left(2 \al^2 + \left[2 (1 - \al) - \al (1 + \al)^2\right] \xf + \al (1 + \al)^2 \xf^2 \right)}{2 (\al - (1 + \al) \xf)^2}.
\end{gather*}
These formulae agree with those presented in Exercise~5.7 of Daley and Gani~\cite{DG} (in which the reader is asked to obtain $\vn$ by making use of Kendall's Principle of Diffusion of Arbitrary Constants).
\end{exa}

\begin{exa}
We define the $\apq$ version of the [MT] model in a similar way, with the rules (a) and (c) given in Example~\ref{E: apq DK} and
\begin{enumerate}
\item[(b$^\prime$)] Once a spreader decides to tell the rumour in a directed contact with somebody already informed, only this spreader chooses with probability~$\al$ to become a stifler
\end{enumerate}
instead of (b) in Example~\ref{E: apq DK}, holding independently for each directed contact between two individuals.
The basic [MT] model has $\al = p = q = 1$.

Since the $\apq$-[MT] model is obtained by the choice
\begin{equation*}
\la = p, \, \de = q, \, \tu = 0, \, \td = \ga = \al \, \text{ and } \, \te = 0,
\end{equation*}
we conclude that the limiting proportion of ignorants is given by
\[ \xf = \xf(q, \al, 0) = - \dfrac{\LW{- h \, e^{-h}}}{h} \text{ with } h = 1 + \frac{q}{\al}. \]
Two particular cases are:

\smallskip
\textbf{(1)} For the $(1, 1, q)$-[MT] model, the values of $\xf$ and $\uf$ are the same as those of the $(1, 1, q)$-[DK] model, given by~\eqref{F: xu 11q DK}.
However,
{\allowdisplaybreaks
\begin{align*}
\Si_{11} &= \frac{\xf (1 - \xf) \left( 1 - (1 + q^2) \xf \right)}{(1 - (1 + q) \xf)^2}, \\[0.1cm]
\Si_{12} &= - \frac{\xf \, \uf^2}{(1 - (1 + q) \xf)^2}, \\[0.1cm]
\Si_{22} &= \frac{\uf \left( q + (1 - 5 q) \xf + (-1 + 4 q + q^2) \xf^2 \right)}{(1 - (1 + q) \xf)^2}.
\end{align*}}%

\smallskip
\textbf{(2)} The $(\al, 1, 1)$-[MT] model corresponds to the basic [MT] model in which the number of contacts with an informed individual by each spreader waiting to become a stifler has geometric distribution with parameter $\al$.
In this case, $\xf$ coincides with that of the $(\al, 1, 1)$-[DK] model, but
\[ \vn = \frac{\xf (1 - \xf) \left(\al^2 - \left(\al^2 + 2 \al - 1\right) \xf \right)}{(\al - (1 + \al) \xf)^2}. \]
This generalized form of the [MT] model in which a spreader becomes a stifler only after being involved in a random number of unsuccessful communications is further studied in Lebensztayn et al.~\cite{RPRS}.
\end{exa}

\begin{exa}
Other rumour models in the literature for which our results apply were proposed by Pearce~\cite{Pearce} and Kawachi~\cite{Kawachi}. 
The first one has the dynamics of the [DK] model, with the following interaction rules. 
A meeting between an ignorant and a spreader results in the ignorant turning into a spreader with probability~$p$. 
When two spreaders interact, either both become stiflers with probability $q_2$ or only one of them does so with probability $q_1 \leq 1 - q_2$. 
Finally, the interaction of a spreader with a stifler results in two stiflers with probability $r$. 
Thus, this model is obtained by considering $\la = p$, $\de = 1$, $\tu = {q_2}/{p}$, $\td = {q_1}/{(2 p)}$ and $\ga = {r}/{p}$. 

One of the deterministic models studied in Kawachi~\cite{Kawachi} evolves similarly to Hayes' model, as explained in Example~\ref{E: Hayes}. 
The contacts are as in the [MT] model, but, when two spreaders meet, both become stiflers with probability~$\tilde{\beta}$; the transformation of only one of them into a stifler is not possible. 
When a spreader encounters an ignorant, the first one transmits the rumour with probability~$\tilde{\alpha}$ and the latter joins the spreaders with probability~$\tilde{\theta}$. 
Finally, in a meeting between a spreader and a stifler, the first individual becomes a stifler with probability~$\tilde{\gamma}$. 
This model is obtained by the choice $\la = \tilde{\alpha}$, $\de = \tilde{\theta}$, 
$\tu = 2 \, \tilde{\beta} / \tilde{\alpha}$, $\td = 0$ and $\ga = \tilde{\gamma} / \tilde{\alpha}$.
\end{exa}

\section{Proofs}
\label{S: Proofs}

The main method for proving Theorems~\ref{T: WLLN} and~\ref{T: CLT} is, by means of a random time change, to define a new process $ \{ \tilde V^{(N)}(t) \}_{t \geq 0} $ with the same transitions as $ \{ V^{(N)}(t) \}_{t \geq 0} $, so that they terminate at the same point.
This transformation is done in such a way that $ \{ \tilde V^{(N)}(t) \}_{t \geq 0} $ is a density dependent Markov chain to which we can apply Theorem~11.4.1 of Ethier and Kurtz~\cite{MPCC}.
The arguments are similar to those used in Kurtz et al.~\cite{EMCG} for the coverage of a random walks system on the complete graph.

\subsection{Random time change}
\label{SS: Random time change}

We define
\begin{align*}
\Theta^{(N)}(t) &= \int_0^t Y^{(N)}(s) \, ds, \, 0 \leq t \leq \tau^{(N)}, \\
\Upsilon^{(N)}(s) &= \inf \{t: \Theta^{(N)}(t) > s \}, \, 0 \leq s \leq \int_0^{\infty} Y^{(N)}(u) \, du,
\end{align*}
and let $ \tilde V^{(N)}(t) = V^{(N)}(\Upsilon^{(N)}(t)) $.
The time-changed process $ \{ \tilde V^{(N)}(t) \}_{t \geq 0} $ has the same transitions as $ \{ V^{(N)}(t) \}_{t \geq 0} $, hence if we define
\[ \tilde{\tau}^{(N)} = \inf \{t: \tilde{Y}^{(N)}(t) = 0 \}, \]
we have that $V^{(N)}(\tau^{(N)}) = \tilde{V}^{(N)}(\tilde{\tau}^{(N)})$. 
Furthermore, $ \{ \tilde V^{(N)}(t) \}_{t \geq 0} $ is a continuous-time Markov chain with initial state $(N, 0, 1)$ and transition rates given by
{\allowdisplaybreaks
\begin{equation}
\label{F: Rates TCP}
\begin{array}{cc}
\text{transition} \quad &\text{rate} \\[0.1cm]
\ell_0 = (-1, 0, 1) \quad &\la \, \de \, \tilde X, \\[0.1cm]
\ell_1 = (-1, 1, 0) \quad &\la (1 - \de) \, \tilde X, \\[0.1cm]
\ell_2 = (0, 0, -2) \quad &\la \, \tu \displaystyle\frac{\tilde Y - 1}{2}, \\[0.2cm]
\ell_3 = (0, 0, -1) \quad &\la \, \td \, (\tilde Y - 1) + \la \, \ga \, (N + 1 - \tilde X - \tilde Y).
\end{array}
\end{equation}}%

\subsection{Deterministic limit of the time-changed process}
\label{SS: Deterministic limit}

We define for $t \geq 0$,
\[ \tilde v^{(N)}(t) = \dfrac{\tilde V^{(N)}(t)}{N} = (\tilde x^{(N)}(t), \tilde u^{(N)}(t), \tilde y^{(N)}(t)), \]
and consider
\begin{align*}
\beta_{\ell_0} (x, u, y) &= \la \, \de \, x, & \beta_{\ell_1} (x, u, y) &= \la (1 - \de) \, x, \\
\beta_{\ell_2} (x, u, y) &= \la \, \tu \, \frac{y}{2}, & \beta_{\ell_3} (x, u, y) &= \la \, \td \, y + \la \, \ga \, (1 - x - y).
\end{align*}
Note that the rates in~\eqref{F: Rates TCP} can be written as
\[ N \left[ \beta_{\ell_i} \left( \dfrac{\tilde X}{N}, \dfrac{\tilde U}{N}, \dfrac{\tilde Y}{N} \right) + O \left( \dfrac{1}{N} \right) \right], \]
so $ \{ \tilde v^{(N)}(t) \}_{t \geq 0} $ is a density dependent Markov chain with possible transitions in the set $\{ \ell_0, \ell_1, \ell_2, \ell_3 \}$.

Now we use Theorem~11.2.1 of Ethier and Kurtz~\cite{MPCC} to conclude that the time-changed system converges almost surely as $ N \to \infty $ (on a suitable probability space).
The drift function is given by
\[ F(x, u, y) = \sum_i \ell_i \, \beta_{\ell_i} (x, u, y) 
= (- \la x, \la (1 - \de) x, \la (\ga + \de) x - \la \te y - \la \ga), \]
hence the limiting deterministic system is governed by the following system of ordinary differential equations
\begin{equation*}
\begin{cases}
x^{\prime}(t) = - \la \, x(t), \\[0.1cm]
u^{\prime}(t) = \la (1 - \de) \, x(t), \\[0.1cm]
y^{\prime}(t) = \la (\ga + \de) \, x(t) - \la \te \, y(t) - \la \ga, \\[0.1cm]
x(0) = 1, u(0) = 0 \text{ and } y(0) = 0.
\end{cases}
\end{equation*}
The solution of this system is given by $ v(t) = (x(t), u(t), y(t)) $, where
\begin{equation*}
x(t) = e^{-\la t}, \, u(t) = (1 - \de)(1 - x(t)), \, y(t) = f(x(t)),
\end{equation*}
with $f$ replaced by $f_{\te}$ if $\te$ equals $0$ or $1$.
According to Theorem~11.2.1 of Ethier and Kurtz~\cite{MPCC},

\begin{lem}
\label{L: Conv v}
We have that $ \tilde v^{(N)}(t) $ converges almost surely to $ v(t) $, uniformly on bounded time intervals.
\end{lem}

We prove that for each of the first two components of $ \tilde v^{(N)} $, the convergence is uniform on the whole line.

\begin{lem}
\label{L: Conv xu}
We have that $ \tilde x^{(N)}(t) $ converges almost surely to $ x(t) $, uniformly on~$ \bbR $.
The analogous assertion holds for $ \tilde u^{(N)} $ and $u$.
\end{lem}

\begin{proof}
We prove the first statement. 
Given any $ \eps > 0 $, we take $t_0 = t_0(\eps)$ such that
$ x(t) \leq \eps/2 $ for all $ t \geq t_0 $.
By the uniform convergence of $ \tilde x^{(N)}(t) $ to $ x(t) $ on the interval $ [0, t_0] $, there exists $N_0 = N_0(\eps)$ such that 
\[ |\tilde x^{(N)}(t) - x(t)| \leq \eps/2 \text{ for all } N \geq N_0 
\text{ and } t \in [0, t_0]. \]
Therefore, for all $ N \geq N_0 $ and $ t \geq t_0 $,
\[ \tilde x^{(N)}(t) \leq \tilde x^{(N)}(t_0) \leq x(t_0) + \eps/2 \leq \eps, \]
whence $ |\tilde x^{(N)}(t) - x(t)| \leq \eps $. \qquad
\end{proof}

\subsection{Proofs of Theorems \ref{T: WLLN} and \ref{T: CLT}}
\label{SS: Proofs of Theorems}

Both theorems follow from Theorem~11.4.1 of Ethier and Kurtz~\cite{MPCC}.
We adopt the notations used there, except for the Gaussian process $V$ defined on p.~458, that we would rather denote by $\GV = (\GV_x, \GV_u, \GV_y)$.
Here $\varphi(x, u, y) = y$, and
\[ \tf = \inf \{t: y(t) \leq 0 \} = - \frac{1}{\la} \, \log \xf. \]
Moreover, from~\eqref{F: Ineq xf},
\begin{equation}
\label{F: Der Neg}
\nabla \varphi(v(\tf)) \cdot F(v(\tf)) = y^{\prime} (\tf) = \la (\ga + \de) \, \xf - \la \ga < 0.
\end{equation}

Let us explain the argument leading to the proof of the Law of Large Numbers.
Although $y(0) = 0$, we have that $y^{\prime}(0) > 0$.
This and~\eqref{F: Der Neg} imply that $y(\tf - \eps) > 0$ and $y(\tf + \eps) < 0$ for $0 < \eps < \tf$.
The almost sure convergence of~$\tilde{y}^{(N)}$ to~$y$ uniformly on bounded intervals yields that
\begin{equation}
\label{F: Conv tf}
\lim_{N \to \infty} \, \tilde \tau^{(N)} = \tf \quad \text{almost surely}.
\end{equation}
Thus, recalling that $V^{(N)}(\tau^{(N)}) = \tilde{V}^{(N)}(\tilde{\tau}^{(N)})$, we obtain Theorem~\ref{T: WLLN} from formula~\eqref{F: Conv tf} and Lemma~\ref{L: Conv xu}.

With respect to the Central Limit Theorem, we get from Theorem~11.4.1 of Ethier and Kurtz~\cite{MPCC} that 
$\sqrt{N} \, (\tilde x^{(N)}(\tilde \tau^{(N)}) - \xf, \tilde u^{(N)}(\tilde \tau^{(N)}) - \uf)$
converges in distribution as $N \to \infty$ to
\begin{equation}
\label{F: LD}
\left( \GV_x(\tf) - A \, \GV_y(\tf), \, \GV_u(\tf) + A \, (1 - \de) \, \GV_y(\tf) \right),
\end{equation}
where $A$ is the constant defined in Theorem~\ref{T: CLT}.
The asymptotic distribution is a mean zero bivariate normal distribution, so it remains to explain how formula~\eqref{F: Cov Matrix} for the covariance matrix~$\Si$ is obtained.

For this, we summarize the steps taken to compute $ \La = \Cov(\GV(\tf), \GV(\tf)) $, a task that can be better carried out with mathematical software.
First, we calculate the matrix of partial derivatives of the drift function $ F $ and the matrix $ G $, as defined in Ethier and Kurtz~\cite[p.~458]{MPCC}.
They are given by
\begin{equation*}
\begin{gathered}
\partial F(x, u, y) = 
\begin{pmatrix} 
-\la & 0 & 0 \\ 
\la (1 - \de) & 0 & 0 \\ 
\la (\ga + \de) & 0 & -\la \te
\end{pmatrix}
\, \text{ and } \\[0.1cm]
G(x, u, y) = 
\begin{pmatrix}
\la x & -\la (1 - \de) x & -\la \de x \\
-\la (1 - \de) x & \la (1 - \de) x & 0 \\
-\la \de x & 0 & \la (\de - \ga) x + \la (\ka - \te + 2 \ga) y + \la \ga
\end{pmatrix}.
\end{gathered}
\end{equation*}
Next, we obtain the solution $\Phi$ of the matrix equation
\[ \frac{\partial}{\partial t} \, \Phi(t, s) = \partial F(x(t), u(t), y(t)) \, \Phi(t, s),
\quad \Phi(s, s) = I_3, \]
where $I_3$ is the $3 \times 3$ identity matrix.
Then, we compute 
\begin{equation}
\label{F: Covariance}
\Cov(\GV(t), \GV(t)) = \int_0^{t} \Phi(t, s) \, G(x(s), u(s), y(s)) \, {[\Phi(t, s)]}^T \, ds.
\end{equation}

We emphasize that the computation of~$\La$ must be separated into four cases: $\te \in (0, 1) \setminus \{ 1/2 \}$, $\te = 1/2$, $\te = 0$ and $\te = 1$.
The value $\te = 1/2$ must be considered separately from the interval $(0, 1)$ owing to the appearance of the integral $\int_0^{t} e^{\la (2 \te - 1) s} \, ds$ in the element $(3, 3)$ of the matrix given in formula~\eqref{F: Covariance}.
This also explains why the constant~$D$ in Theorem~\ref{T: CLT} is defined differently for $\te = 1/2$.

The final step to obtain $\La$ is to set $t = \tf$ in the formula obtained from~\eqref{F: Covariance}.
This is accomplished by making suitable substitutions in this formula according to the value of~$\te$; for instance, for $\te \in (0, 1) \setminus \{ 1/2 \}$ we replace $e^{-\la t}$ and $e^{-\la \te t}$ respectively by
\[ \xf \, \text{ and } \, \frac{(\ga + \de) \, \te \, \xf + \ga (1 - \te)}{\ga + \de \te}. \]
The resulting formula (valid for any $\te$) is
\begin{equation}
\label{F: Final Cov}
\La = 
\begin{pmatrix}
\xf \, (1 - \xf) & -(1 - \de) \, (1 - \xf) \, \xf & 0 \\
-(1 - \de) \, (1 - \xf) \, \xf & (1 - \de) \, (1 - \xf) \, ((1 - \de) \, \xf + \de) & -B \\
0 & -B & D
\end{pmatrix}.
\end{equation}
Using~\eqref{F: LD}, \eqref{F: Final Cov} and the well-known properties of the variance and covariance, we get formula~\eqref{F: Cov Matrix}.

\section*{Acknowledgments}

We thank the referees for the careful reading of the paper and helpful suggestions.
We are grateful to Tom Kurtz and Alexandre Leichsenring for earlier fruitful discussions.

\end{document}